\documentclass{article}

\usepackage{amsmath, amsthm, amssymb,mathtools}
\usepackage{tikz}
\usepackage{tikz-cd}
\usepackage{hyperref}
\newtheorem{pprop}{Proposition}
\newtheorem{teo}[pprop]{Theorem}
\newtheorem{plem}[pprop]{Lemma}
\newtheorem{pclaim}{Claim}

\theoremstyle{definition}
\newtheorem{defi}[pprop]{Definition}
\newtheorem{pnot}[pprop]{Notation}

\title{Groups Definable in Presburger Arithmetic}
\author{Juan Pablo Acosta L\'opez}
\date{}
\begin{document}
\maketitle
\begin{abstract}
Here we give a complete list of the groups definable in Presburger arithmetic
 up to a finite index subgroup.
\end{abstract}
\section{Introduction}
In \cite{definable-bounded-groups-in-Z}, 
it is proven that 
groups definable in Presburger arithmetic are abelian-by-finite.
Additionally, bounded groups are described completely.
The main theorem here is Theorem \ref{main-def-in-Z-2}, which
describes completely all groups definable in Presburger
arithmetic.

We give now an overview of the proof. We prove a version of the Ellis-Nakamura Lemma
which works simultaneously for semigroups 
over definable bounded sets of Presburger arithmetic and over profinite topological spaces.
The statement of the hypothesis uses the concept of a pro-definable set.
This is used to find an idempotent in a space of types obtained from a definable group.
In this idempotent the associativity constraint forces the group structure to be 
the familiar $(Z^r,+)$. This group can be recovered globally (that is, not just on the type)
 using differences in the
given definable group in a process similar to the groupification,
so we obtain an injective group map $Z^r\to G$ with
bounded cokernel, see Theorem \ref{main-def-in-Z}. 
Finally we use group cohomology to calculate the isomorphism type of the group extension,
by choosing a complete type, where the cycle associated to the extension 
is an affine function, 
so the cycle condition is seen to be trivial,
 and then use differences and an observation about
extensions of local group morphisms to global group morphisms on $\mathbb{R}$ 
to obtain a group section sufficient to determine $G$.
\section{The degree of unboundedness}
In what follows $Z$ will denote an $\omega$-saturated $Z$-group.
We will use the shorthand 
$x\ll y$ for $x,y\in Z$ to mean $0<x<y$ and $nx<y$ for all 
$n\in \mathbb{Z}_{>0}$. If $a$ is a finite tuple then
$|a|={\rm Max}_i|a_i|$. If $x\in Z^r$ then
$a\ll x$ means 
$|a|\ll x_1\ll x_2\ll\dots\ll x_r$.
We shall call a $x\in Z^r$ divisible if $x\in nZ^r$ for all $n\in\mathbb{Z}_{>0}$.
More generally
for $p\in\hat{\mathbb{Z}}^r$ and $x\in Z^r$ define
$x\equiv p$ or 
$\hat{x}=p$ if and only if $x_i\equiv p_i \mod N$ for every $N\in\mathbb{Z}_{>0}$,
under the isomorphism $Z/NZ\cong \hat{\mathbb{Z}}/N\hat{\mathbb{Z}}
\cong \mathbb{Z}/N\mathbb{Z}$.

The letter $t$ usually denotes a tuple with components in $Z$, which will
usually be taken to be parameters. 
For instance if $t\in T\subset Z^s$ is a $0$-definable set, 
a family $\{S_t\}_{t\in T}$ of sets $S_t\subset Z^r$ is called $0$-definable
if $S=\bigcup_{t\in T} S_t\times \{t\}\subset Z^r\times T$ is $0$-definable.
If $X$ is a $t$-definable set, say it is defined by a formula 
$\phi(x,t)$, we sometimes denote $X=X_t$ where $X_{t'}$ is the set
defined by the formula $\phi(x,t')$. 
Note that $\{X_t\}_{t\in Z^s}$ is a $0$-definable family.

We note that for Presburger arithmetic there is a quantifier elimination and a 
notion of dimension
of definable sets see \cite{cdZ} for the definition and properties.

We shall make use of cell decomposition for $Z$, see
\cite{cdZ}. We provide the statement of this theorem next.
\begin{defi}
We define the notion of an 
$a$-definable cell $C\subset Z^n$ by induction
on $n$.
If $C\subset Z^n$ is an $a$-definable cell
then $D\subset Z^{n+1}$ is said to be an $a$-definable cell if
it is of one of the following forms, first
$D=\{(x,y)\in Z^n\times Z \mid x\in C, 
f(x)<y<g(x), y\equiv k\mod N\}$
for a $N\in\mathbb{Z}_{>0}$ and $k\in\mathbb{Z}/N\mathbb{Z}$, and
$f,g:C\to Z$ given by $f(x)=A_1x+c_1$
and $g(x)=A_2x+c_2$, for $A_i,B_i$ rational
matrices, $c_i$ $a$-definable constants, and $C$ satisfying
conditions of divisibility that make the image of $f$ and $g$ 
belong to $Z$, and also such that
$g(x)-f(x)>N$.
We also accept as cells those where $f$ or $g$ do not appear.
\end{defi}
This definition is somewhat different from the usual one, which
requires $f-g$ to be unbounded, or equal to $2$, but for our purposes
it is enough.

\begin{pprop}
If $A_{1},\dots,A_{r}$ are $a$-definable sets which
form a partition of $B$ and $f:B\to Z^n$ is an $a$-definable function
then there exists an $a$-cell decomposition of $B$
$C_{1},\dots,C_{s}$
such that $C_{i}$ refine $A_{j}$ and 
$f$ restricted to each $C_{i}$ is of the form
$x\mapsto Ax+c$ for $A$ a matrix with rational coefficients
and $c$ an $a$-definable constant.
\end{pprop}

Now we define an invariant $r={\rm ubd}(X)$ for every definable set $X$,
the ``degree of unboundedness'' to be the greatest $r\in\mathbb{N}$ such that
there exists $f:Z^r\to X$ injective. We prove later that 
$X$ is bounded
if and only if ${\rm ubd}(X)=0$.

We start with a simple remark
\begin{plem}\label{ubd+}
For a definable set $X$, ${\rm ubd}(X)$ is also the greatest $r$ such that there is an
injective definable function $Z_{>0}^r\to X$. 
In fact, $Z$ is in definable bijection with $Z_{>0}$.
\end{plem}
\begin{proof}
This follows by considering even and odd
numbers. More precisely $f:Z\to Z_{>0}$ is defined as $f(x)=2x$ if $x>0$ and 
$f(x)=1-2x$ if $x\leq 0$.
\end{proof}
\begin{plem}\label{ubdlema2}
Let $c$ be a tuple of elements in $Z$.
Let $C$ be a definable set containing the 
set of $x\in Z^r$ such that $c\ll x$ and $x$ is divisible.
Then there exists an injective definable map $Z^r\to C$.
\end{plem}
\begin{proof}
By compactness there exists $d\in Z_{>0}$, 
$n_1,\dots,n_{r-1}\in\mathbb{Z}_{>0}$, and $N\in \mathbb{Z}_{>0}$
such that $C$ contains the set of $x$ such that $d<x_1, n_1x_1<x_2,\dots, n_rx_{r-1}<x_r$,
and $x_i\in NZ$. 
Then the map $f:Z_{>0}^r\to C$ given by $x\mapsto y$ where
$y_1=Nd+Nx_1$ and $y_i=n_iy_{i-1}+Nx_i$, is injective definable. We conclude
with Lemma \ref{ubd+}.
\end{proof}
\begin{plem}\label{ubdlema3}
Let $c$ be a finite tuple in $Z$.
If $X$ is the set of elements $x\in Z^r$ such that $c,1\ll x$ and $x$ is divisible, then
$X$ is the set of realizations of a complete type over $c$. In particular if
$f:X\to Z^s$ is $c$-definable then there is a constant $d$ and a rational matrix 
$A$ with $f(x)=Ax+d$ for all $x\in X$.
\end{plem}
\begin{proof}
By elimination of quantifiers we have to see that if $x,y\in X$ then for every
$n\in\mathbb{Z}^r$ and $m\in\mathbb{Z}^s$, $l,N\in \mathbb{Z}$, we have
$nx+mc+l>0$ if and only if $ny+mc+l>0$ and $nx+mc+l\equiv ny+mc+l\mod N$.
The second follows because $x,y$ are divisible.
The first follows trivially if $n=0$. If $n\neq 0$ then $n_kx_k$ is the dominant term
of $nx+mc+l$, if $n_k$ is the first nonzero term of $n$ from right
to left. So $nx+mc+l>0$ is equivalent to $n_k>0$, and similarly for $y$, as required.
\end{proof}
\begin{plem}\label{ubdlema}
If $X\subset Z^t$ 
is a definable set, then there exist $X_i$ disjoint definable sets such that
$X=X_1\cup\dots\cup X_p$ and for each $i$ there exists $r_i\in\mathbb{Z}_{\geq 0}$,
a bounded cell $C_i$ with $C_i\subset Z^{t-r_i}$, and a bijection
$f_i:Z_{>0}^{r_i}\times C_i\to X_i$, which are of the form
$f_i(x)=A_ix+b_i$ for $A_i$ matrices with integer coefficients and $b_i$
tuples in $Z^t$.
\end{plem}
\begin{proof}
This proof is by induction on $t$.
The base case is formally $t=0$ which is trivial. Alternatively the proof
of the induction step establishes $t=1$.

So assume the lemma is true for $t$ and we prove it for $t+1$.
Let $X\subset Z^{t+1}$ be a definable set. Using cell decomposition we may assume
$X$ is a cell. In other words, there is a cell $Y\subset Z^t$ such that
$X=\{(a,z)\mid a\in Y, g(a)<z<f(a), z\equiv s\mod N\}$, for affine functions $g,f$ such that $g(a)+N<f(a)$ for all $a\in Y$,
or where $f$ or $g$ do not appear.

Applying the bijection $z\mapsto Nz+s$ we may assume $N$ does not appear in the definition
of $X$. If $M$ is the largest denominator of the rational numbers which 
appear as coefficients of $f,g$; then after 
dividing $X$ into $M^t$ sets according to the residue of $a$ mod $M$, and
applying the bijections 
$a\mapsto Ma+i$ for $i\in [0,M)^t$ we may further assume that $f,g$ have integer 
coefficients.

Now we apply the induction hypothesis to $Y$ to assume $Y=Z_{>0}^r\times B$, for
a bounded cell $B\subset Z^{t-r}$.
We write $X=\{(a,b,z)\mid g(a,b)<z<f(a,b), a\in Z_{>0}^r, b\in B\}$.

If $f$ does not appear but $g$ does we use the bijection $(a,z,b)\mapsto (a,b,z+g(a,b))$
from $Z_{>0}^{r+1}\times B$ to $X$.

If $g$ does not appear but $f$ does we use the bijection 
$(a,z,b)\mapsto (a,b,f(a,b)-z)$ from $Z_{>0}^{r+1}\times B$ to $X$.

If neither $f$ and $g$ appear then we divide $X$ into the sets with $z>0$ and 
$z<1$. Both of these we know how to handle.

So now we assume both $f$ and $g$ appear.
Using the function $(a,b,z)\mapsto (a,b,z+g(a,b))$ we may assume $g=0$.
Write $f(a,b)=na+mb+c$ where $n,m$ are $1\times t$ 
matrices with integer coefficients, in other words, they are tuples.
If $n=0$, then $X=Z_{>0}^r\times B'$ with $B'$ a bounded cell of the form
required.

Assume $n\neq 0$. Then after a permutation in the $Z_{>0}^r$ factor we assume $n_1>0$. 
Dividing by $n_1$ and taking classes mod $n_1$ as in the first reduction we may assume $n_1=1$. 
Denote $a=(a_1,a')$ and $n=(1,n')$. Now we divide $X$ into two disjoint 
subsets $X_1$ and $X_2$ defined as follows.
A tuple $(a,b,z)\in X_1$ if $a\in Z_{>0}^r$, $b\in B$, $z\in Z_{>0}$ 
and $z-n'a'-mb-c\leq 0$.
A tuple $(a,b,z)\in X_2$ if $a\in Z_{>0}^r$, $b\in B$, $z\in Z_{>0}$ 
and $0<z-n'a'-mb-c<a_1$.
For $X_1$ we have that $X_1=Z_{>0}\times T$ for $T\subset Z^t$ so we are done
by induction.
For $X_2$ we have the bijective affine map $(a',z,b,s)\mapsto (z-n'a'-mb-c+s,a',b,z)$,
from $R\times Z_{>0}$ to $X_2$. 
Where $R=\{(a',z,b)\mid a'\in Z_{>0}^{r-1}, b\in B, z\in Z_{>0}, 
0<z-n'a'-mb-c\}$ is a definable set such that
 $R\subset Z^{t}$, so here we also are done by induction.
\end{proof}
\begin{pprop}\label{ubdmain}
If $X$ is a definable set then ${\rm ubd}(X)\leq r$ if and only if $X$ is in definable bijection with some definable set
$X'\subset Z_{>0}^r\times [0,a)^s$.
\end{pprop}
\begin{proof}
Assume first ${\rm ubd}(X)\leq r$.
By the Lemma \ref{ubdlema} we have $X=X_1\cup \dots \cup X_p$, $r_i$, $C_i$ and $f_i$ as in
the statement of the lemma. By Lemma \ref{ubd+} we have that $r_i\leq r$. From this the
result follows.

Assume now that $X\subset Z_{>0}^r\times [0,a)^s$. We want to see that 
${\rm ubd} (X)\leq r$.
Assume now towards a contradiction
that $f:Z^{r+1}\to Z^{r}\times [0,a)^s$ is injective.
If $\pi:Z^r\times [0,a)^s\to [0,a)^s$ is the projection, then there is a constant $b$
and a matrix $A$ with rational coefficients such that
$\pi f(x)=Ax+b$ for all $c,1 \ll x$ divisible, where $c$ are defining parameters for $f$, see
Lemma \ref{ubdlema3}.
But as $f$ has a bounded image, then
necessarily $A=0$. So $\pi f$ is constant on a cell containing such $x$
which, by Lemma \ref{ubdlema2},
contains an injective image of $Z^{r+1}$.
We obtain an injective $g:Z^{r+1}\to Z^r$. This is not possible for dimension
reasons.
\end{proof}
\begin{pprop}\label{ubd}\strut
\begin{enumerate}
\item $X\subset Z^r$ has ubd 0, if and only if $X$ is bounded.
\item ${\rm ubd}(X)\leq {\rm dim}(X)$. 
\item ${\rm ubd}(Z^r)=r$.
\item If $X=Y\cup Z$ then ${\rm ubd}(X)={\rm Max}\{{\rm ubd}(Y),{\rm ubd}(Z)\}$
\item ${\rm ubd}(X\times Y)={\rm ubd}(X)+{\rm ubd}(Y)$.
If $H\leq G$ are groups then 
${\rm ubd}(G)={\rm ubd}(H)+{\rm ubd}(G/H)$.
\item If $a$ is a tuple with components in $Z$, and 
$\{X_r\}_r$ is an $a$-definable family, then the set
$\{r\mid {\rm ubd}(X_r)=n\}$ is $a$-definable.
\end{enumerate}
\end{pprop}
\begin{proof}
1) If $X$ is not bounded, by 
definable choice there is definable
$f:Z_{\geq 0}\to X$ such that $f(t)\in X\setminus [-t,t]^r$.
By cell decomposition there are $a\in Z$ and $N\in\mathbb{Z}_{>0}$
such that if $x>a$ and $N|x$, then 
$f(x)=mx+b$. If $m=(m_1,\dots,m_n)$ then $m_i\neq 0$ for all $i$.
Then $g(x)=f(a+Nx)$ is injective from $Z_{> 0}\to X$.
We conclude by Lemma \ref{ubd+}.

\medskip
 
2) If ${\rm ubd}(X)\geq r$, then there is an injective definable map
$Z^r\to X$, so by properties of dimension 
$r={\rm dim}(Z^r)\leq {\rm dim}(X)$.

\medskip

3) It is clear from the definition that ${\rm ubd}(Z^r)\geq r$.
From item 2) we get the other inequality.

\medskip

4) It is clear from the definition that
${\rm ubd}(X)\geq{\rm Max}\{{\rm ubd}(Y),{\rm ubd}(Z)\}$.
By the previous proposition we get the other inequality.

\medskip

5) Similarly, from the definition
 ${\rm ubd}(X\times Y)\geq {\rm ubd}(X)+{\rm ubd}(Y)$,
and the previous proposition gives the other inequality.

By definable choice
$G\cong H\times G/H$ as a definable set.

\medskip

6) The set $A_n$ defined by $r\in A_n$ if and only if ${\rm ubd}(X_r)\geq n$
is $a$-$\lor$-definable, this follows from the definition. Indeed, if $r\in A_n$, then
there is an $atr$-definable injective function $f_{atr}:Z^n\to X_r$, where $t$ are some
defining parameters,
so $A_n$ contains the $a$-definable set of those $r'$ such that there exists
$t'$ such that
$f_{ar't'}$ is an injective definable function $Z^n\to X_{r'}$.
Similarly from the previous proposition the set $B_n$ of $r$ such
that ${\rm ubd}(X_r)\leq n$ is $a$-$\lor$-definable.
So $C_n=A_n\cap B_n$ is also $a$-$\lor$-definable.
If $X_r\subset Z^N$, then
$C_n$ is also the intersection of the complements of $C_s$ for
$0\leq s\leq N$ and $s\neq n$, so it is also $a$-$\land$-definable.
So by compactness $C_n$ is $a$-definable.
\end{proof}
\section{Idempotents in prodefinable semigroups}
As motivation for Lemma \ref{idempotent2} recall that a nonempty 
compact Hausdorff
topological semigroup has an idempotent, (this is the 
Ellis-Nakamura Lemma, see \cite{idempotent} Lemma 1),
 and because having an idempotent is first order expressible 
a nonempty bounded
definable semigroup in $(Z,+,<)$ also has one.
The hypothesis of Lemma \ref{idempotent2} includes both situations.

For the Lemma \ref{idempotent2} we need some terminology.
A reference for this material is \cite{prodef}, but with a small variant,
we keep track of the cardinality of the index set
for saturation reasons.

The category of $A$-$\kappa$-pro-definable sets is the category
with objects
the diagrams of $A$-definable
sets $X_i$ indexed by a directed set $i\in I$ with intermediate maps
$X_i\to X_j$ $A$-definable, and 
the cardinality of $I$ $\leq \kappa$. 
This object is denoted $X=\lim_i X_i$. 
For the $A$-$\kappa$-pro-definable sets $X=\lim_i X_i$ and $Y=\lim_j Y_j$ the set of morphisms 
of $A$-$\kappa$-pro-definable sets is 
${\rm Mor}(X,Y)=\lim_j{\rm colim}_i{\rm Mor}(X_i,Y_j)$.
We will denote this category ${\rm ProDef}_{A,\kappa}$.
There is an inclusion functor ${\rm Def}_A\to {\rm ProDef}_{A,\kappa}$
which is fully faithful.

Assume that the model $M$ is $\kappa$-saturated in the language 
$L(A)$.
Then there is a functor $F:{\rm Def}_A\to {\rm Sets}$ that takes
the $A$-definable set $D$, to the set of $M$ points of $D$.
This functor extends in the obvious manner to
a forgetful functor
 $F:{\rm ProDef}_{A,\kappa}\to {\rm Sets}$.
By compactness this functor is faithful.
This functor factors through $F:{\rm ProDef}_{A,\kappa}\to {\rm Top}$ 
by giving $F(X)$ the topology given by basic open sets the
$A$-relatively definable subsets of $F(X)$, in other words
the inverse image of $A$-definable subset of $X_i$ via the canonical
projection $F(X)\to X_i$.
We remark that because the category ${\rm Def}_A$ has nonempty finite limits
and the forgetful functor ${\rm Def}_A\to {\rm Sets}$ commutes with them
then ${\rm ProDef}_{A,\kappa}$ has limits indexed by nonempty categories
of cardinality $\leq \kappa$, and 
${\rm ProDef}_{A,\kappa}\to {\rm Sets}$ commutes with them.
However ${\rm ProDef}_{A,\kappa}\to {\rm Top}$ does not commute with limits.

Recall that a topological space is T$_0$ if two points belonging to the same open sets
are equal. The T$_0$-ification of a topological space is the space obtained by identifying
two points if they belong to the same open sets.

For an $A$-$\kappa$-pro-definable set $X$, $F(X)$ is a compact topological
space and its T$_0$-ification is Hausdorff. The 
forgetful functor ${\rm ProDef}_{A,\kappa}\to {\rm Top}$ and its
composition with the T$_0$-ification functor commute with 
directed limits with index set of cardinality $\leq \kappa$.

Given an $A$-$\kappa$-pro-definable set 
$X$, an $A$-$\kappa$-pro-definable subset $Y$ is a subobject
$Y\to X$ isomorphic to the one
given by compatible $A$-definable subsets $Y_i\subset X_i$.
These subobjects are determined by the image of
the inclusion $F(Y)\to F(X)$ in ${\rm Sets}$.
Indeed if $Z$ is an $A$-$\kappa$-pro-definable object
${\rm Mor}(Z,Y)={\rm Mor}(Z,X)\times_{{\rm Mor}(F(Z),F(X))}{\rm Mor}(F(Z),F(Y))$. 
The inclusion $F(Y)\to F(X)$ is a closed immersion.
Similarly an $A$-relatively definable subset of $X$ is a subobject
$Y\to X$ isomorphic to a subobject of the form $\lim_{j\geq i}\pi_{ji}^{-1}Y_i$ for some index $i$ and 
$A$-definable subset $Y_i\subset X_i$. These subobjects correspond to
open and closed subsets $Y\subset F(X)$. 
So in particular $A$-type-definable subsets of an $A$-definable set
$D$
are $A$-pro-definable subsets of $D$.

\begin{pprop}\label{pro-def-full}
\begin{enumerate}
\item If $X,Y$ are $A$-$\kappa$-pro-definable sets
and $f:F(X)\to F(Y)$ is such that for any $A$-definable set $E$,
$f\times 1:F(X\times E)\to F(Y\times E)$ is continuous, then
there is $g:X\to Y$ such that $F(g)=f$.
\item The forgetful functor $F:{\rm ProDef}_{A,\kappa}\to {\rm Sets}$ reflects isomorphisms, that is, if $f:X\to Y$ is such that
$F(f)$ is a bijection, then $f$ is an isomorphism.
\end{enumerate}
\end{pprop}
The first one follows from
the description of pro-definable morphisms in Proposition 8
of \cite{prodef}, the second one is proven in Proposition 8 of
\cite{prodef}.

Given a theory with two $0$-definable constants, finite sets
can be considered as $0$-definable sets. The profinite
topological space $X$ is considered as a $0$-pro-definable
set via $X=\lim_{U}X/U$ where the limit is over
the finite partitions of $X$ by clopens, $X/U=U$ and the canonical
projection $\pi_U:X\to X/U$ is given by $x\mapsto R$ where
$x\in R$ and $R\in U$.

We note that for a $0$-type-definable set
$L\subset M^n$, the map $L\to S^n(0)$ is $0$-pro-definable.
\begin{plem}\label{idempotent2}
Let $M$ be a sufficiently saturated model of a theory $\mathcal{T}$ which defines two 
different $0$-definable constants.

Suppose $D$ is a $0$-definable set.
Suppose $X$ is a profinite topological space.
Suppose $C_d\subset Z^r, d\in D$, 
is a $0$-definable family of sets.

Assume that there exists an elementary submodel $N$ of $M$ 
 such that for every $d\in D(N)$, $C_d$ is finite.

Suppose given $Y_d\subset X\times C_d$ for $d\in D$, $Y_d$ a nonempty semigroup.
We consider $X$ as a $0$-pro-definable set.
Denote $Y=\cup_d Y_d\times\{d\}$ and 
$C=\cup_d C_d\times\{d\}$
Assume:
\begin{enumerate}
\item $Y$ as a subset of $X\times C$
 is $0$-relatively definable.
\item The given semigroup product $Y\times_DY\to Y$ 
is a $0$-pro-definable
map.
\end{enumerate}
Then $Y_d$ has an idempotent.
\end{plem}
We use the previous lemma in the case where $M$ is an $\omega$-saturated model
of Presburger arithmetic, $N=\mathbb{Z}$ and $C_d$ are bounded sets.
\begin{proof}
Denote $\pi_U:X\times C\to X/U\times C$ the canonical projection.

Hypothesis 1 says that there is 
$U_0$ and $F_0\subset X/U_0\times C$ $0$-definable such that
 $Y=\pi_{U_0}^{-1}(F_0)$.

Denote $F_U=\pi_{U,U_0}^{-1}F_0$ for 
$U$ refining $U_0$ and $\pi_{U,U_0}:X/U\times C\to X/U_0\times C$
the canonical projection.

Denote $E_U\subset D$ defined by $d\in E_U$ if and only if $Y_d$ has a
$U$-idempotent, that is, there is a $(x,c)\in Y_d$ such that
$\pi_U(x,c)^2=\pi_U(x,c)$.

Hypothesis 2 says that for every $U$ refining $U_0$ there is $V$ refining $U$
such that $Y\times_D Y\to Y\to F_U$ factors as
$Y\times_D Y\to F_V\times_D F_V\to F_U$. Denote
$m_{V,U}:F_V\times_D F_V\to F_U$.
Then $d\in E_U$ if and only if there exists $x$ such that $(x,d)\in F_V$
and $m_{V,U}((x,d),(x,d))=\pi_{V,U}(x,d)$. We conclude that
$E_U$ is $0$-definable.

Take now $d\in D\cap N^s$, where $D\subset M^s$. Then $C_d$ is finite, so
$Y_d$ is a non-empty compact Hausdorff topological semigroup,
so by the Ellis-Nakamura Lemma $d\in E_U$.
We have seen $D\cap N^s=E_U\cap N^s$.
But $E_U$ is $0$-definable, so $D=E_U$. 
We conclude that if $d\in D$ then $Y_d$ has a $U$-idempotent, say
$(x_U,c_U)$.

Take $(x,c)\in Y_d$ an accumulation point of the net
$(x_U,c_U)$ in $Y_d$ with respect to the $d$-pro-definable
topology. If $Z_{d,V}$ is the set of $V$-idempotents of $Y_d$
then $Z_{d,V}$ is the pullback of $\Delta:F_{V,d}\to F_{V,d}^2$
(the diagonal) along $Y_d\to Y_d^2\to F_{V,d}^2$, where the first 
arrow is $x\mapsto (x^2,x)$. As $\Delta$ is an injective map
of $d$-definable sets we see that $Z_{d,V}$ is $d$-relatively
definable. So for $V$ fixed, because $(x_U,c_U)$ is in $Z_{d,V}$
for eventual $U$, we conclude $(x,c)\in Z_{d,V}$. That is, $(x,c)$ is 
$V$-idempotent for every $V$. This implies that $(x,c)$ is an
idempotent.
\end{proof}
Note that for a $0$-pro-definable family (as in 1) and 2)) the set
of $d$ such that $Y_d$ is a semigroup is $0$-type-definable.
That it contains a $0$-definable set around the point of interest
is part of the hypothesis.

We use the next Lemma only for $M,\Sigma$ and $C_d$ as 
in the situation \ref{types-at-infinity}.
In that situation $M$ is a $Z$-group, 
$C_d$ is bounded and $b\vDash \Sigma_a$ informally means ``$b$ is infinite
relative to $a$", see also condition 2) of Lemma \ref{definability}.
Condition 1) is something like orthogonality.
\begin{plem}\label{definability}
Let $M$ be a model of any theory that defines
two different $0$-definable constants, which is sufficiently saturated.

Let $\Sigma(x,y)$ be a partial type with parameters in $0$, and
variables in $M^n\times M$.

Denote $\Sigma_b(x)=\cup_i\Sigma(x,b_i)$ for $b=(b_1,\dots,b_m)$.

Take $D$ a $0$-definable set and 
$\{C_d\}_{d\in D}$ a $0$-definable family.

Suppose given a $0$-definable family $\{S_d\}_{d\in D}$ of 
left cancellative semigroups, $S_d\subset M^n\times C_d$.

Define $X\subset S^n(0)$ the set of types $p$ such that for every
$t\in M$ there is $a\vDash p\cup \Sigma_t$. Then $X$ is a profinite
topological space.

Assume that:
\begin{enumerate}
\item If $a\vDash \Sigma_t$ and $a'\vDash \Sigma_{t'}$, and
$a\equiv_0 a', t\equiv_0 t'$, then $at\equiv_0 a't'$.
\item If $a\vDash \Sigma_{dt}$ with $d\in D$, and $c\in C_d$, then
$a\vDash \Sigma_{dtc}$.
\item If $f:L\to M^n$ is an injective 
$a$-definable function and $p\in X$,
$b\vDash p\cup\Sigma_a$, $b\in L$, then $f(b)\vDash \Sigma_a$.
\end{enumerate}

Denote $Y_d\subset S^n(0)\times C_d$ to be the set of pairs
$(p,c)$ such that $p\in X$ and there is $a\vDash p\cup \Sigma_d$ 
such that
$(a,c)\in S_d$.

Define $Y\times_D Y\to Y$ by
$(p_1,c_1,d)\cdot(p_2,c_2,d)=(p_3,c_3,d)$ if
there is $a_1\vDash p_1\cup\Sigma_d$ and 
$a_2\vDash p_2\cup\Sigma_{da_1}$ such that
$(a_1,c_1)(a_2,c_2)=(a_3,c_3)$ and
$a_3\vDash p_3\cup\Sigma_{d}$.

Then $Y$ is $0$-relatively definable in $X\times C$
 and $Y\times_D Y\to Y$ is a well
defined $0$-pro-definable map, that makes $Y_d$ a semigroup.

We also have that if 
$a_1\vDash p_1\cup\Sigma_d$ and 
$a_2\vDash p_2\cup\Sigma_{da_1}$,
$(a_1,c_1)(a_2,c_2)=(a_3,c_3)$ and
$a_3\vDash p_3\cup\Sigma_{d}$, then $a_3\vDash \Sigma_{da_1}$.

If in addition $d_0\in D$ is such that there exists
$h:T\to S_{d_0}$ an injective $d_0t$-definable map and there exists
$p_0\in X$ such that $p_0\cup \Sigma_{d_0t}\vdash T$, then
there exists $d_0\in D_0\subset D$ $0$-definable such that
$Y_d$ is non-empty for all $d\in D_0$.
\end{plem}
\begin{proof}
We start with the observation that $X\subset S^n(0)$ is closed.
Indeed $X$ is the image of a closed subset of $S^n(M)$ under the
scalar restriction map $S^n(M)\to S^n(0)$. As these are 
compact Hausdorff spaces, $X$ is closed.

Now we note that condition 1) implies that for every 
$p\in S^n(0)$ and $t$, if $p\cup \Sigma_t$ is consistent, then it is
a complete type over $t$. 

Denote $R\subset M^n\times C$ the set of tuples $(a,c,d)$ such that
$tp(a/0)\in X$ and $a\vDash \Sigma_d$. This is a $0$-type-definable
set. Take now the map $R\to X\times C$ given by $(a,c,d)\mapsto
(tp(a/0),c,d)$. This is a $0$-pro-definable surjective map
$R\to X\times C$. 
\begin{pclaim}\label{bijective-on-t0}
This map $R\to X\times C$ is bijective in T$_0$-ifications, and so it is
an homemorphism in T$_0$-ifications.
\end{pclaim}
\begin{proof}
This is a consequence of conditions 1) and 2).

Indeed if $(a,c,d)$ and $(a',c',d')$ map to the same element in the 
T$_0$-ification of $X\times C$, this means 
$tp(a/0)=tp(a'/0)$ and $cd\equiv_0 c'd'$. From the definition of $R$ we have
$a\vDash \Sigma_d$ and $c\in C_d$; and similarly
 $a'\vDash \Sigma_{d'}$ and $c'\in C_{d'}$.
From condition 2) we have $a\vDash \Sigma_{cd}$ and $a'\vDash \Sigma_{c'd'}$.
From condition 1) we have $acd\equiv_0 a'c'd'$. In other words 
$(a,c,d)$ and $(a',c',d')$ map to the same element in T$_0$-ification of $R$.
\end{proof}
We conclude that
the image of the open and closed set $\Theta_1\subset R$ given by
$(a,c,d)\in\Theta_1$ if and only if $(a,c,d)\in R$ and $(a,c)\in S_d$ is $Y$
which is then open
and closed in $X\times C$, in other words $Y$ is $0$-relatively 
definable in $X\times C$.

Now define $\Theta_2\subset S\times_D S$ and 
$\Theta_3\subset S\times_DS\times_DS$ by 
$(a_1,c_1,d,a_2,c_2,d)\in \Theta_2$ if and only if
$a_1\vDash \Sigma_d$, $tp(a_1/0)\in X$, $a_2\vDash \Sigma_{a_1d}$ and $tp(a_2/0)\in X$.
Similarly $(a_1,c_1,d,a_2,c_2,d,a_3,c_3,d)\in \Theta_3$
if and only if 
\[tp(a_i/0)\in X, a_1\vDash \Sigma_d,a_2\vDash\Sigma_{da_1},a_3\vDash\Sigma_{da_1a_2}.\]

Then there are $0$-prodefinable surjections 
\begin{center}$\Theta_2\to Y\times_D Y$
and $\Theta_3\to Y\times_DY\times_DY$\end{center} given by 
\begin{center}
$(a_1,c_1,d,a_2,c_2,d)\mapsto (tp(a_1/0),c_1,d,tp(a_2/0),c_2,d)$
and $(a_1,c_1,d,a_2,c_2,d,a_3,c_3,d)\mapsto (tp(a_1/0),c_1,d,tp(a_2/0),c_2,d,tp(a_3/0),c_3,d)$.
\end{center}
These maps become bijective on T$_0$-ifications with an argument similar to
claim \ref{bijective-on-t0}.

Now if $m:S\times_D S\to S$ is the multiplication, then we will show
that $m(\Theta_2)\subset \Theta_1$, $1\times m(\Theta_3),m\times 1(\Theta_3)\subset \Theta_2$, and the maps $m$ factors through 
$\Theta_2\to Y\times_DY, \Theta_1\to Y$.
This would show that the product $Y\times_DY\to Y$ defined in the statement
is well-defined and associative. The map $\cdot$ is then continuous. In fact
$(\cdot)\times 1_E$ are also continuous for any $E$ $0$-definable
(as this has the effect of replacing $D$ by $D\times E$). So the product
$\cdot$ is a map
of $0$-pro-definable sets, see Proposition \ref{pro-def-full} 1).

\begin{pclaim}\label{well-defined-theta-1}
Let $a_1\vDash \Sigma_{dt},a_2\vDash \Sigma_{dta_1},
(a_1,c_1)\in S_d$, $(a_2,c_2)\in S_d$ and $(a_1,c_1)(a_2,c_2)=(a_3,c_3)$,
then $a_3\vDash \Sigma_{dta_1}$. 
\end{pclaim}
\begin{proof}
Indeed $a_3=Y(a_1,a_2,c_1,c_2,d)$
and $c_3=Z(a_1,a_2,c_1,c_2,d)$
for $0$-definable maps $Y$, $Z$. By hypothesis 2) we have 
$a_2\vDash \Sigma_{dta_1c_1c_2c_3}$. Now if
$a_2'\equiv_{dta_1c_1c_2} a_2$, then 
$a_2'\vDash \Sigma_{dta_1c_1c_2}$ and so by hypothesis 2) we get
$a_2'\vDash\Sigma_{dta_1c_1c_2c_3}$ and 
$a_2'\equiv_{dta_1c_1c_2c_3}a_2$. From this it follows that
$Z(a_1,a_2',c_1,c_2,d)=Z(a_1,a_2,c_1,c_2,d)=c_3$, and so, because 
$S_d$ is left cancellative $Y(a_1,a_2',c_1,c_2,d)\neq Y(a_1,a_2,c_1,c_2,d)$
if $a_2\neq a_2'$.
Now by hypothesis 3) we conclude $a_3\vDash \Sigma_{dta_1}$.
\end{proof}
This proves the claim after the definition of $Y$ in the statement of 
this lemma.
\begin{pclaim}\label{well-defined-theta-2}
Let $a_1,a_2,a_3,c_1,c_2,c_3$ be as in the previous claim.
If we take $a_1', a_2'$ such that
$a_1'a_2'\equiv_{c_1c_2dt}a_1a_2$ then $a_1'a_2'\equiv_{dta_1c_1c_2c_3} a_1a_2$
and $Z(a_1',a_2',c_1,c_2,d)=Z(a_1,a_2,c_1,c_2,d)=c_3$.
\end{pclaim}
The proof is similar to the previous claim.

These two claims show that 
$m(\Theta_2)\subset \Theta_1$ and that it factors as a product
$Y\times_D Y\to Y$. We also have $1\times m(\Theta_3)\subset \Theta_2$.

We are left with proving that $m\times 1(\Theta_3)\subset \Theta_2$.

\begin{pclaim}\label{well-defined-theta-3}
 $tp(a/0)\in X$ and $a\vDash \Sigma_b$ then
$a\vDash \Sigma_{bc}$ for $c\in{\rm dcl}(b)$. 
\end{pclaim}
\begin{proof}
Indeed if $a'$ is such that
$a'\equiv_0 a$ and $a'\vDash \Sigma_{bc}$, then $a'b\equiv_0 ab$ and as 
$c\in {\rm dcl}(b)$ we get $a'bc\equiv_0 abc$ so $a\vDash \Sigma_{bc}$.
\end{proof}
Now let $(a_1,c_1,d,a_2,c_2,d,a_3,c_3,d)\in\Theta_3$ and 
$(a_1,c_1)(a_2,c_2)=(a_4,c_4)$. 
Note that $a_4=Y(a_1,a_2,c_1,c_2,d)$ for a $0$-definable
function $Y$. We have seen in claim \ref{well-defined-theta-1} 
that $a_4\vDash \Sigma_{d}$. 
Now as $a_3\vDash\Sigma_{a_1a_2d}$ we get
$a_3\vDash \Sigma_{a_1a_2c_1c_2d}$ and so, by claim 
\ref{well-defined-theta-3}, $a_3\vDash \Sigma_{a_4d}$. 
This shows that $m\times 1(\Theta_3)\subset \Theta_2$ as required.

Now for the last statement, we consider
\[D_1=\{d\in D\mid \text{There exists }t, p_0\cup\Sigma_{dt}\vdash T_{dt},
\text{ and }, h_{dt}:T_{dt}\to S_d \text{ is injective }\}.\]
By compactness this is a 
$0$-$\lor$-definable set so we take $d_0\in D_0\subset D_1$
$0$-definable.
Now take $d\in D_0$ and $t$ be such that  $p_0\cup\Sigma_{dt}\vdash T_{dt},$
and $h_{dt}:T_{dt}\to S_d$  is injective. Take $s$ some extra parameters. 
If $a\vDash p_0\cup\Sigma_{dts}$,
then $h_{dt}(a)=(a'',c)$ with $a''=Y(a,d,t), c=Z(a,d,t)$ , $Y$, $Z$, 
$0$-definable functions. 
As in the proof of claim \ref{well-defined-theta-1} we have that if
$a'\vDash p_0\cup\Sigma_{dts}$ and $a'\neq a$, then 
$Z(a,d,t)=Z(a',d,t)$, $Y(a,d,t)\neq Y(a',d,t)$ and so
$Y(a,d,t)\vDash \Sigma_{dts}$. So
for $p=tp(a''/0)$ we get $(p,c)\in Y_d$.
This ends the proof of the lemma.
\end{proof}
\begin{pnot}\label{types-at-infinity}
Let $(Z,+,<)$ be an $\omega$-saturated $Z$-group.

Define $\Sigma(x,y)$ to be the partial type over $0$ with variables in
$Z^r\times Z$, given by
$(a,b)\vDash \Sigma$ if and only if for every nonzero tuple $m\in\mathbb{Z}^r$
$|b|,1\ll |ma|$.
For $b\in Z^k$ denote $\Sigma_b(x)=\bigcup_i\Sigma(x,b_i)$ where $b=(b_1,\dots,b_k)$.
Denote $X\subset S^r(0)$ the set of types $p$ such that, $a\vDash p$ implies
$1\ll |ma|$ for every nonzero tuple $m\in \mathbb{Z}^r$.
Here $ma$ denotes $\sum_i m_ia_i$.
\end{pnot}
\begin{plem}\label{type-lema-z}
With the notation of \ref{types-at-infinity}.
If $b\vDash \Sigma_{ac}$ and $c$ is divisible then $b\equiv_a b+c$.
\end{plem}
\begin{proof}
By elimination of quantifiers we have to see that $nb+ma+r>0$ if and only if 
$n(b+c)+ma+r>0$ and that $nb+ka+r\equiv n(b+c)+ma+r\mod N$ for tuples $n,m,r$ with integers
coefficients and $N$ an integer.
The second condition follows because $c$ is divisible.
The first condition is trivial if $n=0$. If $n\neq 0$, by definition of $\Sigma$ we get
that $nb$ is the dominant term in both expressions, so they are equivalent to $nb>0$.
\end{proof}
\begin{plem}\label{disintegration}
Let $Z$ be a $\omega$-saturated $Z$-group.
Let $\{C_d\}_{d\in D}$ be a $0$-definable family of bounded
sets.
Define $\Sigma(x,y)$ as in \ref{types-at-infinity}.
Then $\Sigma$ and $C_d$ satisfy hypothesis 1,2 and 3 of Lemma
\ref{definability}
\end{plem}
\begin{proof}
Let $a,b,a',b'$ be such that $a\vDash \Sigma_b$, $a'\vDash \Sigma_{b'}$
and $a\equiv_0 a', b\equiv_0 b'$. By elimination of quantifiers we have
to see that $na+mb+k>0$ if and only if $na'+mb'+k>0$ and 
$na+mb+k\equiv na'+mb'+k\mod N$. If $n\neq 0$, then 
$na+mb+k>0$ if and only if $na>0$. If $n=0$, this becomes a condition on $b$.
The divisibility type of $na+mb+k$ depends only on the divisibility
type of $a$ and $b$. This proves Hypothesis 1.

If $f(d)={\rm Max}\{|x|\mid x\in C_d\}$, then $f$ is $0$-definable
and so there is a cell decomposition of $D$ such that on each of the
cells $f$ is of the form $Ad+B$. So $|f(d)|\leq N(|d|+1)$
for an integer $N$. This shows Hypothesis 2.

Let $f,a,b,p$ as in the statement of Hypothesis 3.
Then there exists $A,B,C$ with rational entries such that
$f(b)=Ab+Ba+C$. 

We claim that $A$ is invertible
Otherwise if $A$ is singular then we first find a non-zero divisible 
tuple $c$ such that $b\vDash \Sigma_c$ and $Ac=0$. This exists because for any $N$
we may find $d$ non-zero with integer coefficients, and so satisfying $b\vDash \Sigma_d$,
 with $Ad=0$ and $c\equiv 0 \mod N$, so we conclude by compactness.
Note however that $b+c\equiv_a b$ by Lemma \ref{type-lema-z}
So now we conclude that
$f(b)=f(b+c)$.
This contradicts the injectivity of $f$.

So $A$ is invertible. If $m\neq 0$, then 
$|mf(b)|\geq |mAb|-|mBa|-|C|\gg |a|,1$ so 
$f(b)\vDash \Sigma_a$.
\end{proof}
For convenience we 
specialize the previous lemmas to the situation at hand in the 
lemma that follows
\begin{plem}\label{idempotent3}
Let $(Z,+,<)$ be an $\omega$-saturated $Z$-group.

Let $T$ be a $0$-definable set.
Let $\{C_t\}_{t\in T}$ be a $0$-definable family of bounded sets.
Let $\{G_t\}_{t\in T}$ be a $0$-definable family of groups, such that
$G_t\subset Z^r\times C_t$. 

Assume ${\rm ubd}(G_t)=r$ for every $t$.

With notation as in \ref{types-at-infinity}, 
for every $t$ there exists $p\in X$ and $b\in C_t$
such that if
$x\vDash p\cup \Sigma_t$ and $y\vDash p\cup\Sigma_{(t,x)}$,
then $(x,b)$ and $(y,b)$ belong to $G_t$ and 
$(x,b)(y,b)=(z,b)$ satisfies $z\vDash p\cup \Sigma_{(t,x)}$.
\end{plem}
\begin{proof}
We take $Y$ as in Lemma \ref{definability}, with $G$ taking the place of $S$.
The hypothesis of that Lemma is satisfied by Lemma \ref{disintegration}.
We conclude by Lemma \ref{idempotent2}.
\end{proof}
\section{Groups definable in Presburger arithmetic}
\begin{plem}\label{eventual-injectivity-implies-injectivity}
Let $\Sigma$ and $X$ be as in \ref{types-at-infinity}.
If $f:(\cap_n nZ)^r\to G$ is a type-definable group morphism such that
$f$ is injective in 
$p\cup \Sigma_t$ for some
$t\in Z$ and $p\in X$ that implies $(\cap_n nZ)^r$, then $f$ is injective.
\end{plem}
\begin{proof}
If $K$ is the kernel of $f$ and $K\neq 0$, then if we take
$y\in K$ nonzero, divisible, and take $x\vDash p\cup \Sigma_{ty}$ 
then $x+y\vDash p\cup \Sigma_{ty}$ by Lemma \ref{type-lema-z}, so 
we conclude $f(x)=f(x+y)$ which contradicts the injectivity of $f$ in 
$p\cup\Sigma_t$.
\end{proof}
\begin{teo}\label{main-def-in-Z}
If $G$ is an abelian group definable in $(Z,+,<)$, then 
$G$ has a subgroup $H$ isomorphic as a definable group to $Z^n$ such that
$G/H$ is bounded.
\end{teo}
\begin{proof}
Suppose $G\subset Z_{>0}^r \times [0,a)^s$ with $r={\rm ubd}(G)$, as in
Proposition \ref{ubdmain}. Let $t$ be a set of parameters over which $G$ is 
defined, so that
$G$ is a $t$-definable group, and without loss of generality $a$ is in $t$.
Say $G=G_t$, for $t\in T$. 
We may assume that the projection onto the last $s$ factors of 
$G_{t'}$ is bounded for all
$t'\in T$, and that ${\rm ubd} (G_{t'})=r$ for all $t'\in T$.
To distinguish it from other sums we will denote the product of $G$ by
$\oplus$ and the inverse by $\ominus$.

Take $\Sigma$ and $X$ as in the situation \ref{types-at-infinity}.
By Lemma \ref{idempotent3} we have that there exist
a type $p\in X$ and an element $b\in [0,a)^s$ 
such that if $x\vDash p\cup\Sigma_t$
and $y\vDash p\cup\Sigma_{(t,x)}$ then $(x,b),(y,b)\in G$ and 
$(x,b)\oplus(y,b)=(z,b)$ satisfies $z\vDash p\cup\Sigma_{(t,x)}$. 

We shall give a type-definable function 
$i:q\cup\Sigma_t\to G$ defined on a
type $q\in X$ that implies $\bigcap_n (nZ)^r$, 
and such that $i(x+y)=i(x)\oplus i(y)$ for
$x\vDash q\cup\Sigma_t$ and $y\vDash q\cup\Sigma_{(t,x)}$.
Start with $j:p\cup\Sigma_t\to G$, $j(x)=(x,b)$.
Denote $\hat{p}$ the element in $\hat{\mathbb{Z}}^r$ such that
if $x\vDash p$ then $\hat{x}=\hat{p}$.
We can write $j(x)\oplus j(y)=(Ax+By+d,b)$ where $A,B$ are 
matrices with rational entries and $d$ is a $tb$-definable constant.
From the identity 
$(j(x)\oplus j(y))\oplus j(z)=j(x)\oplus (j(y)\oplus j(z))$
for all $x\vDash p\cup\Sigma_t, y\vDash p\cup\Sigma_{(t,x)},
z\vDash p\cup\Sigma_{(t,x,y)}$,
we obtain $A^2x+ABy+Bz+Ad+d=Ax+BAy+B^2z+Bd+d$.
Then $A^2=A, B^2=B, BA=AB, Ad=Bd$.
As $G$ is a group it is left and right cancellative so 
$A$ and $B$ must be
nonsingular rational matrices, so $A=B=1$.
Reducing mod $\hat{\mathbb{Z}}$ obtain $d\equiv -\hat{p}$,
indeed from $(x,b)\oplus (y,b)=(x+y+d,b)$, for 
$x\vDash p\cup \Sigma_{t}$ and $y\vDash p\cup\Sigma_{(t,x)}$ we have that
$x+y+d\vDash p$ and in particular its reduction mod $\hat{\mathbb{Z}}$
is $\hat{p}$, and on the other hand this reduction is also
$\hat{x}+\hat{y}+\hat{d}=2\hat{p}+\hat{d}$.
Taking $i(x)=j(x-d)$ we obtain the identity
\begin{equation}\label{local-group-homo-eq}
i(x+y)=i(x)\oplus i(y)
\end{equation}
for all $x\vDash q\cup\Sigma_t$ and 
$y\vDash q\cup\Sigma_{(t,x)}$, where
$q=p+d$. Indeed for $x\vDash q\cup \Sigma_t$ and 
$y\vDash q\cup\Sigma_{(t,x)}$ we have
$x=x'+d$ and $y=y'+d$ for $x'\vDash p\cup\Sigma_t$ and 
$y'\vDash p\cup\Sigma_{(t,x')}$, and then we get
$i(x)\oplus i(y)=j(x-d)\oplus j(y-d)=j(x')\oplus j(y')=
(x'+y'+d,b)=(x+y-d,b)=i(x+y)$.

Define $k:(\bigcap_n nZ)^r\to G$ by $k(z)=i(x+z)\ominus i(x)$ 
for $x\vDash q\cup\Sigma_{(t,z)}$. 
One has to prove $k$ is a well defined group morphism.
Suppose $x,y\vDash q\cup\Sigma_{(t,z)}$. Take
$w\vDash q\cup\Sigma_{(t,z,x,y)}$.
Then 
$i(x+z+w)\ominus i(x+w)=i(x+z)\ominus i(x)$, applying 
\ref{local-group-homo-eq} twice.
On the other hand if $x+w=y+w'$ then $w'\vDash \Sigma_{(t,z,x,y)}$
and so, similarly, $i(y+z+w')\ominus i(y+w')=i(y+z)\ominus i(y)$, so
$k$ is well-defined.

Now
if $z,w\in(\bigcap_n nZ)^r$ and
$x\vDash q\cup \Sigma_{(t,z,w)},y\vDash q\cup\Sigma_{(t,z,w,x)}$, then
\[k(z)\oplus k(w)=i(x+z)\oplus i(y+w)\ominus (i(x)\oplus i(y))=i(x+y+z+w)\ominus i(x+y)=k(z+w).\]
 In the first and last equation we use the definition of $k$ and in the second
we use \ref{local-group-homo-eq} twice.
$k$ restricts to $i$ on $x\vDash q\cup\Sigma_t$, so $k$ is injective, by Lemma \ref{eventual-injectivity-implies-injectivity}.
By compactness $k$ extends to an injective definable group morphism
$(NZ)^r\to G$, and composing with multiplication by $N$
we obtain an injective group morphism
$k':Z^r\to G$.
Then $G/k'(Z^r)$ is bounded, by Proposition \ref{ubd}.
\end{proof}
The next lemma is implicit in \cite{definable-bounded-groups-in-Z},
but for convenience we include a proof.

By a $\lor$-definable group we mean a 
$\lor$-definable set $G$ with a group operation $G\times G\to G$
such that the for every $X,Y\subset G$ definable
the product restricted to $X\times Y$
is definable. In other words the product is a map of ind-definable
sets.
\begin{plem}\label{group-morphism-extension-R}
In an arbitrary theory.
Let $G$ be a $\lor$-definable abelian group
 and $T\subset G$ a type-definable subgroup.
 Let $H$ be a $\lor$-definable group
and $\phi:T\to H$ a type-definable group morphism.
Suppose given a surjective group morphism
$\pi:G\to \mathbb{R}^s$ with kernel $T$. Assume that:
\begin{enumerate}
\item If $T\subset U\subset G$ is definable then there is 
$0\in V\subset \mathbb{R}^s$ open such that
$\pi^{-1}(V)\subset U$.
\item If $T\subset U\subset G$ is definable then $\pi(U)$ is bounded.
\end{enumerate}
Then $\phi$ extends to a unique group morphism $\phi:G\to H$
such that the restriction of $\phi$ to a definable subset is definable.
\end{plem}
\begin{proof}
First uniqueness.
If $T\subset U_0\subset G$ is definable and $\phi_1,\phi_2:G\to H$
extend $\phi$, then $\phi_i|_{U_0}$ are definable and equal when 
restricted to $T$. By compactness there is $T\subset U_1\subset U_0$
definable with $\phi_1|_{U_1}=\phi_2|_{U_1}$.
By 1) $\cup_{n\geq 0}nU_1=G$, so $\phi_1=\phi_2$.

Now existence.
Choose $U$ definable symmetric such that $\phi$ extends
to a function $\phi:U\to H$ such that $\phi(x+y)=\phi(x)\phi(y)$
for all $x,y\in U$.
Replacing $H$ by the group generated by $\phi(U)$ we may assume
$H$ is abelian, and we denote it additively.

Take $0\in B$ a convex open such that $\pi^{-1}B\subset U$.
For $N\in\mathbb{Z}_{>0}$ define
$A_N=\{x\in G\mid \text{ there exists }y\in U, z\in T, Ny+z=x\}$.
Then $A_N\subset (N+1)U$ and by compactness $A_n$ is type-definable.
Define $\phi_N:A_N\to H$ by $\phi_N(x)=N\phi(y)+\phi(z)$.
We see this is well defined, if 
$Ny_1+z_1=Ny_2+z_2$ then $N\pi(y_1)=N\pi(y_2)$ so
$y_2=y_1+v$ for $v\in T$. From $Ny_1+z_1=Ny_2+z_2$ obtain
$z_1=Nv+z_2$. Then $\phi(y_2)=\phi(y_1)+\phi(v)$ and
$\phi(z_1)=N\phi(v)+\phi(z_2)$, from where
$N\phi(y_2)+\phi(z_2)=N\phi(y_1)+\phi(z_1)$, which is what we 
wanted to see.
By compactness $\phi_N$ is type-definable.
If $x$ is such that $\frac{1}{M}\pi(x)\in B$ then for $N,N'\geq M$,
$x\in A_N\cap A_{N'}$ and
$\phi_N(x)=\phi_N'(x)$.
Indeed, considering $(N,NN'), (N',NN')$ without loss
$N'=NN_1$. Take $y\in G$ and $z\in T$ such that
$x=NN_1y+z$. Then 
 $y,2y,\dots,N_1y\in \pi^{-1}(B)\subset U$, so 
$\phi(N_1y)=N_1\phi(y)$, and $\phi_N(x)=\phi_{N'}(x)$, as required.
Define $\psi(x)=\phi_N(x)$ for $N$  sufficiently large,
(Which clearly extends $T\to H$).
Now if $x,y\in G$ and $N$ is sufficiently large
then $x=Nx'+z_1, y=Ny'+z_2$ and $x+y=N(x'+y')+z_1+z_2$
with $x',y',x'+y'\in U$ and $z_1,z_2\in T$. From here
$\psi(x+y)=\psi(x)+\psi(y)$.
Finally if $X$ is definable and $\frac{1}{N}\pi(X)\subset B$, then
$\psi=\phi_N$ on $X$ so we have the definability.
\end{proof}
We recall now one of the main results of \cite{definable-bounded-groups-in-Z}.
To state the result we need some notation. 
For $a\in Z$ with $1\ll a$ denote
\[O(a)=\{x\in Z\mid \text{ there exists }n\in\mathbb{Z}_{>0},|x|<na\}\]
and
\[o(a)=\{x\in Z\mid \text{ for all }n\in\mathbb{Z}_{>0}, n|x|<a\}.\]
For $a\in Z^s$ with $1\ll a_i$, denote
$O(a)=O(a_1)\times\cdots\times O(a_s)$
and $o(a)=o(a_1)\times\cdots\times o(a_s)$.
As remarked in
\cite{one-dimensional-p-adic}, Proposition 3.2,
 $O(a)/o(a)\cong \mathbb{R}^s$, and this isomorphism
is easily seen to satisfy the hypotheses of the Lemma \ref{group-morphism-extension-R}.
Recall that a subgroup $\Lambda\subset \mathbb{R}^s$ is called
a lattice if it is discrete with the subspace topology. Equivalently
if it is generated as a group by linearly independent elements of the
$\mathbb{R}$-vector space $\mathbb{R}^s$.
The lattice is called full if the $\mathbb{R}$-linear span is
$\mathbb{R}^s$, equivalently if it is generated as a group by
a basis.
A subgroup $\Lambda\subset O(a)$ is called a local lattice
if $\Lambda\cap o(a)=0$ and $\pi(\Lambda)\subset \mathbb{R}^s$ is a
full lattice. In this case $\Lambda=\bigoplus_i \mathbb{Z}b_i$ 
for $\{\pi(b_i)\}_i$ which form a basis of $\mathbb{R}^s$.
Note that $\Lambda$ satisfies that:
\begin{enumerate}
\item $X\cap \Lambda$ is finite
for all $X\subset O(a)$ definable
\item there exists $X_0$ definable
such that $X_0+\Lambda=O(a)$.
\end{enumerate}
Indeed, $\pi(X)$ is compact and
$X_0$ is any such that $o(a)\subset X_0$ and $\pi(X_0)$ contains
a parallelogram for $\pi(\Lambda)$.
Note that $\pi(\Pi_i[-Na_i,Na_i])$ contains $[-N,N]^s$ so this $X_0$ exists.
This implies that we can consider $O(a)/\Lambda$ as a definable 
group. Indeed in $X_0$ there is the equivalence relation given
by the fibers of $X_0\to O(a)/\Lambda$ which is definable by 1.
By definable choice there is $X_1\subset X_0$ definable
such that the restriction of $O(a)\to O(a)/\Lambda$ to $X_1$ 
is bijective. The sum on $X_1$ which makes this bijection 
a group isomorphism is definable by 1.
Indeed this sum $\oplus$ is defined by $x\oplus y=x+y-z$ where 
$z\in\Lambda\cap 3(X_1\cup -X_1)$ is the unique element of 
$\Lambda$ such that $x+y-z\in X_1$. 

The group $O(a)/\Lambda$ considered as a definable group is denoted
$C(a,b)$ (for $\Lambda=\bigoplus_i\mathbb{Z}b_i$).

In the next theorem we also need the groups of the form
$Z^r\times O(a)/\bigoplus_i\mathbb{Z}(v_i,b_i)$ for $v_i\in Z^r$ and 
$\sum_i\mathbb{Z}b_i$ a local lattice of $O(a)$. If $X_1$ is a definable 
set as found before, that is, a set such that the restriction of 
$O(a)\to C(a,b)$ to $X_1$ is a bijection,
then we consider 
$Z^r\times O(a)/\bigoplus_i \mathbb{Z}(v_i,b_i)$ a definable group with underlying set $Z^r\times X_1$. The sum here is given by 
$(t,x)\oplus (w,y)=(t,x)+(w,y)-(f(z),z)$ where 
$z\in \Lambda\cap 3(X_1\cup -X_1)$ is the unique element of $\Lambda$ such that
$x+y-z\in X_1$ and $f:\Lambda\to Z^r$ is the group morphism that sends $b_i$
to $v_i$.

In \cite{definable-bounded-groups-in-Z} it is proven that
every bounded definable group has a definable finite index subgroup
isomorphic as a definable group to $C(a,b)$ for some $a,b$ as before.
\begin{teo}\label{main-def-in-Z-2}
If $G$ is a group definable in $(Z,+,<)$, then $G$ has a finite 
index definable subgroup isomorphic as a definable group to 
$Z^r\times O(a)/\bigoplus_i\mathbb{Z}(v_i,b_i)$,
for $\sum_i \mathbb{Z}b_i$ a local lattice of $O(a)$ and
$v_i\in Z^r$.
\end{teo}
\begin{proof}
As mentioned before, in \cite{definable-bounded-groups-in-Z}
it is proven that $G$ is abelian-by-finite, so without loss of generality
$G$ is abelian.
By Proposition \ref{main-def-in-Z} we have a short exact
sequence $0\to Z^r\to G\to B\to 0$, with $B$ bounded. 
Denote $i:Z^r\to G$, $q:G\to B$.
By the fact mentioned before this theorem me may take
$B=C(a,b)$. By definable choice we may take
$B\to G$ a set-theoretic definable section.
Denote $s$ the composition $O(a)\to B\to G$.
From this we obtain a definable inhomogeneous 2-cocycle
$g:O(a)^2\to Z^r$ defined by $i(g(x,y))=s(x+y)-s(x)-s(y)$.
This function satisfies 
$g(x,y)+g(x+y,z)=g(y,z)+g(x,y+z)$, called the cocycle condition.

Suppose everything is defined over $t$.
We may assume that $o(a_1)=\cdots=o(a_{n_1})\subsetneq o(a_{n_1+1})
=\cdots=o(a_{n_2})\subsetneq\cdots \subsetneq o(a_{n_{l-1}+1})=\cdots=
o(a_s)$,
with $n_0=0<n_1<\cdots<n_{l-1}<n_l=s$. If $x\in o(a)$ then 
denote $x^i=(x_{n_{i-1}+1},\cdots,x_{n_i})$ for $i=1,\cdots,l$.
Define $W$ to be $x\in W$ if and only if $x\in o(a)$ and
$t'\ll x^i$ for every
$t'$ $t$-definable in $o(a_{n_i})$, so
in particular $a_j\ll x^i$ for 
$j\leq n_{i-1}$,
 and $1\ll x$.
Take $Y$ to be the elements $(x,y)$, $x,y\in o(a)$ such that
$x,y\in W$,
and $x^i\ll y^i$.
 As an aside $Y$ comes from a tensor product of invariant types.
We take $X\subset o(a)\times o(a)$ a complete $t$-type of elements
determined by $(x,y)\in X$ if and only if $(x,y)\in Y$ and $x,y$ are divisible.

Note that as $X$ is a complete type over $t$, 
we can write $g(x,y)=Ax+By+d$ for $(x,y)\in X$, 
for some $A,B$ matrices with rational entries and $d$ a 
$t$-definable constant.
Take $x,y,z\in o(a)$ with $(x,y),(y,z)\in X$.
From the cocycle condition obtain
$Ax=Bz$ from which one gets
$A=0=B$. Replacing $s$ by $s-i(d)$ we may assume $d=0$.
By compactness there is $N$ such that if 
$N|x,y$ and $(x,y)\in Y$ then
 $s(x+y)=s(x)+s(y)$.
Replacing $G$ by $q^{-1}((NO(a)+\Lambda)/\Lambda)$ and
$\Lambda$ by the inverse image under multiplication by $N$, 
$O(a)\to O(a)$, we may assume that $N=1$.
Define $s':o(a)\to G$ by
 $s'(x)=s(x+y)-s(y)$ for $y\in W$, with $|x_i|\ll y_i$.
We see that $s'$ is a well defined group morphism.
For this note that
for $y,y'\in W$ there is $w\in W$ such that $(y,w),(y',w)\in Y$.
The rest follows
as in the last two paragraphs of the proof of
Proposition \ref{main-def-in-Z}.
Note that $s'$ composed with $G\to B$ is the canonical projection
$O(a)\to B$ restricted to $o(a)$.
We obtain now a group morphism
$s'':O(a)\to G$, by Lemma \ref{group-morphism-extension-R}.
By uniqueness of the morphism of that lemma we find that the composition
$O(a)\to G\to B$ equals the canonical projection $O(a)\to B$, because
they restrict to the same mapping on $o(a)$.
Finally if one considers the diagram

\begin{tikzcd}
0\arrow{r} & Z^r\arrow{r} & G\arrow{r}& B\arrow{r} & 0\\
0\arrow{r} & \Lambda\arrow{r}\arrow[u,dotted] & O(a) \arrow{r}
\arrow{u} & B\arrow{r}\arrow{u}{1} & 0,
\end{tikzcd}

then 
there is unique dotted arrow which makes the diagram commute
and a diagram chase shows that the left square is cocartesian,
which is what we had to prove. 
Indeed, the group $\Lambda$ is a free abelian group so a 
map $\Lambda\to Z^r$ is the same thing as choosing some elements in $Z^r$, which
are the image of the basis of $\Lambda$. 
These elements are $-v_i$, with $v_i$ satisfying the statement.
\end{proof}
\bibliographystyle{plain}
\bibliography{definable-group-in-z}
\end{document}